\documentclass[10pt,reqno]{amsart}

\usepackage[a4paper,left=35mm,right=35mm,top=30mm,bottom=30mm,marginpar=25mm]{geometry}

\usepackage{amsmath}
\usepackage{amssymb}
\usepackage{amsthm}
\usepackage[latin1]{inputenc}
\usepackage{eurosym}
\usepackage[dvips]{graphics}
\usepackage{graphicx}
\usepackage{epsfig}
\usepackage{hyperref}
\usepackage{dsfont}

\usepackage[displaymath,mathlines]{lineno}

\allowdisplaybreaks

\usepackage{hyperref}

\usepackage{ifthen}
\makeindex

\renewcommand{\div}{\operatorname{div}}

\newcommand{\Rr}{{\mathbb{R}}}

\newcommand{\Nn}{{\mathbb{N}}}

\newcommand{\Tt}{{\mathbb{T}}}

\newcommand{\Vv}{{\mathcal{V}}}

\def\leq{\leqslant}
\def\geq{\geqslant}

\numberwithin{equation}{section}

\newtheoremstyle{thmlemcorr}{10pt}{10pt}{\itshape}{}{\bfseries}{.}{10pt}{{\thmname{#1}\thmnumber{
#2}\thmnote{ (#3)}}}
\newtheoremstyle{thmlemcorr*}{10pt}{10pt}{\itshape}{}{\bfseries}{.}\newline{{\thmname{#1}\thmnumber{
#2}\thmnote{ (#3)}}}
\newtheoremstyle{defi}{10pt}{10pt}{\itshape}{}{\bfseries}{.}{10pt}{{\thmname{#1}\thmnumber{
#2}\thmnote{ (#3)}}}
\newtheoremstyle{remexample}{10pt}{10pt}{}{}{\bfseries}{.}{10pt}{{\thmname{#1}\thmnumber{
#2}\thmnote{ (#3)}}}
\newtheoremstyle{ass}{10pt}{10pt}{}{}{\bfseries}{.}{10pt}{{\thmname{#1}\thmnumber{
A#2}\thmnote{ (#3)}}}

\theoremstyle{thmlemcorr}
\newtheorem{theorem}{Theorem}
\numberwithin{theorem}{section}
\newtheorem{lemma}[theorem]{Lemma}
\newtheorem{corollary}[theorem]{Corollary}

\newtheorem{teo}[theorem]{Theorem}
\newtheorem{lem}[theorem]{Lemma}
\newtheorem{pro}[theorem]{Proposition}
\newtheorem{cor}[theorem]{Corolary}

\theoremstyle{thmlemcorr*}
\newtheorem{theorem*}{Theorem}
\newtheorem{lemma*}[theorem]{Lemma}
\newtheorem{corollary*}[theorem]{Corollary}
\newtheorem{proposition*}[theorem]{Proposition}
\newtheorem{problem*}[theorem]{Problem}
\newtheorem{conjecture*}[theorem]{Conjecture}

\theoremstyle{defi}

\newtheorem{hyp}{Assumption}

\newtheorem{problem}{Problem}

\theoremstyle{remexample}
\newtheorem{remark}[theorem]{Remark}

\theoremstyle{ass}

\usepackage{color}

\begin{document}

\title[Mean-field games with congestion]{On the existence of solutions for stationary mean-field games with congestion}

\author{David Evangelista}
\address[D. Evangelista]{
        King Abdullah University of Science and Technology (KAUST), CEMSE
        Division, Thuwal 23955-6900. Saudi Arabia.}
\email{david.evangelista@kaust.edu.sa}

\author{Diogo A. Gomes}
\address[D. A. Gomes]{
        King Abdullah University of Science and Technology (KAUST), CEMSE Division, Thuwal 23955-6900. Saudi Arabia.}
\email{diogo.gomes@kaust.edu.sa}


\thanks{
        D. Gomes and D. Evangelista were partially supported baseline and start-up funds from King Abdullah University of Science and Technology (KAUST)
}
\date{\today}

\begin{abstract}

Mean-field games (MFGs)  are models of large populations of rational agents who seek to optimize an objective function that takes into account their location and the distribution of the remaining agents. Here, we consider stationary MFGs with congestion and prove the existence of stationary solutions. Because moving in congested areas is difficult, agents prefer to move in non-congested areas. As a consequence, the model becomes singular near the zero density. The existence of stationary solutions was previously obtained for MFGs 
with quadratic Hamiltonians thanks to a very particular identity. Here, we develop robust estimates that give the existence of a solution for general subquadratic Hamiltonians.

\end{abstract}

\maketitle

\section{Introduction}

Mean-field games (MFGs) constitute a class of mathematical models that capture the behaviors of large populations of competing rational agents. These models combine  differential
games with changes in location of a population of agents. The
agents seek to optimize an objective function by choosing appropriate
controls. In turn,  the aggregated changes of the agents' locations determine the evolution of the distribution of the population. These models comprise a Hamilton-Jacobi
equation coupled with a Kolmogorov-Fokker-Planck equation. Since the seminal works of Lasry and Lions
\cite{ll1}, \cite{ll2}, \cite{ll3}, and Huang, Caines and Malham\'e  \cite{Caines2}, \cite{Caines1}, 
MFG models have been  widely studied by the mathematical and engineering communities.

Here, we study the MFG with congestion as described in Problem \ref{P1} below. To simplify the presentation, we work in the spatially periodic setting; that is, $x \in \Tt^d$, the standard $d$-dimensional torus.
\begin{problem}\label{P1}
        For a  $C^\infty$ Hamiltonian, $H:\Tt^d\times \Rr^d \to \Rr$,
a $C^\infty$ potential, $V:\Tt^d\times \Rr_0^+\to \Rr$, and a congestion
exponent, $\alpha>0$,
        find $u,m:\Tt^d\to\Rr$, $u,m\in C^\infty(\Tt^d)$, $m>0$ solving
\begin{equation}\label{eq1}
\begin{cases}
&u-\Delta u +m^{\alpha}H\left(x,\frac{Du}{m^{\alpha}}\right) + V(x,m) = 0,\\
& m -\Delta m - \div \left( D_pH\left( x,\frac{Du}{m^{\alpha}}\right)m\right)
=1. 
\end{cases}
\end{equation}      
\end{problem}
Before proceeding, we discuss the motivation for this problem.
Let $(\Omega,\mathcal F, P, \mathcal F_t)$ be a filtered probability space, where $\mathcal F_t\subset \mathcal F$ is a filtration and $W_t$ is a $\mathcal F_t$-adapted $d$-dimensional Brownian motion on $[0,+\infty)$. Let $\mathcal V$ denote the set of all $\mathcal F_t$-progressively measurable, $\Rr^d$-valued functions, $v_{t}$, on $\Rr$.  We say that  $v_{t}$ is a control.
The trajectory of each agent is represented by a solution, $X:\Omega\times [0,+\infty)\to \Rr^d$, of a controlled stochastic differential 
equation (SDE),
\begin{equation}
\label{SDE}
\begin{cases}
&dX_t = v_t ds +  \sqrt{2}dW_t,\quad t\in [0,+\infty )\\
&X_t = x.  
\end{cases}
\end{equation}
The probability density, $m:\Tt^d\times[0,+\infty)\to \Rr^+$, gives the spatial distribution of the agents; that is, for any Borel set, $F$, we have $P(X_t\in F) = \int_F m(x,t)dx$.  A classical verification theorem  gives that $v_t=-D_pH\left( x,\frac{Du}{m^{\alpha}}\right)$
is the optimal control. Therefore, the control, $v_t$,
in \eqref{SDE} is Markovian and, thus,  $m$ satisfies a Fokker-Planck equation with the drift vector field $-D_pH\left(x,\frac{Du}{m^\alpha}\right)$.

We fix a Lagrangian, $L:\Tt^d\times \Rr^d \to \Rr$, that determines the  cost of movement and a potential,  $V:\Tt^d\times \Rr^+_0$, that 
 accounts for spatial preferences and interactions of agents. 
Next, we introduce the infinite-horizon,
discounted-cost functional,  \begin{equation*}
J(x,t;v) = \mathbb E \int_t^{\infty}e^{-s}[ m(X_s,s)^\alpha L(X_s,v_s)-V(X_s,m(X_s,s))]ds.
\end{equation*}
Here, $J$\ is the cost for an agent located at $x$\ at time $t$ who is using the control $v_t$.  $J$ takes into account the density around the agent's position as encoded by the term $m(X_s,s)^\alpha L(X_s,v_s).$ The exponent $\alpha$ determines the strength of the congestion.
We assume that each agent is rational and seeks to minimize  $J$. Then, we define the value function, 
\begin{equation*}
u(x,t)=\inf_{v\in \mathcal V} J(x,t;v).
\end{equation*}
The Lagrangian with congestion, $\hat{L}:\Tt^d\times\Rr\times\Rr_0^+\times\Rr^+\to \Rr$, is
\[
\hat{L}(x,v,m,t)=m(x,t)^\alpha L(x,v).
\]
The corresponding Hamiltonian, $\hat{H}:\Tt^d\times\Rr^d\times\Rr_0^+\times\Rr^+\to \Rr$, is given by the Legendre transform of $\hat{L:}$
\[
\hat{H}(x,p,m,t):=\sup_{v\in\mathcal V}\{-v\cdot p - \hat{L} (x,v,m,t) \}=m^\alpha H\left(x,\frac{p}{m^\alpha}\right),
\]
where $H$ is the Legendre transform of $L$. In the next section, we discuss the assumptions about $H$ and the corresponding properties of $L$. Although somewhat technical, these assumptions are natural in optimal control problems, viscosity solutions, and MFGs. The following  Lagrangian is a typical example where our results apply.
Let $\gamma \in (1,2)$,   $\frac 1 \gamma + \frac{1}{\gamma'}=1$, $a\in
C^\infty (\Tt^d)$, $a>0$, and set
\begin{equation}\label{eq:example}
L(x,v) = a(x)(1+|v|^2)^{\frac{\gamma'}{2}}.
\end{equation}
Then, there exists $\gamma_\alpha>1$ such that for $1<\gamma<\gamma_\alpha$, \eqref{eq:example} and the corresponding Hamiltonian satisfy all assumptions detailed in the next section.

Under the assumption of symmetry on the distributions of the agents, we obtain the following time-dependent system: 
\begin{equation}\label{eq1_timedep}
\begin{cases}
&-u_t+u-\Delta u +m^{\alpha}H\left(x,\frac{Du}{m^{\alpha}}\right) + V(x,m) = 0,\\
&m_t+ m -\Delta m - \div {\left( D_pH\left( x,\frac{Du}{m^{\alpha}}\right)m\right)}
=0. 
\end{cases}
\end{equation}    
Typically, this system is supplemented with an initial condition for $m$ and an asymptotic behavior for  $u$. For a detailed discussion of a similar model with a quadratic Hamiltonian, see \cite{ll3} and \cite{LCDF}.

The system  \eqref{eq1} is a stationary version of \eqref{eq1_timedep} with  a unit source on the right-hand side of the second equation of \eqref{eq1}. This source corresponds to an inflow of agents that replace the ones who are leaving due to the discount, the term $m$ in the Fokker-Planck equation.  Without this source, the only non-negative solution to the second equation in  \eqref{eq1} is the trivial  solution, $m=0$. We interpret this stationary model as follows.  The motion of the agents is determined by their optimal controls. Thus, they follow the  SDE \eqref{SDE} with the optimal drift $v_t = -D_pH\left(x,\frac{Du}{m^\alpha}\right)$. The density, $m,$ is an invariant distribution for this process with an additional discount (agents leave $\Tt^d$ at a rate $1$) and an additional source. In turn, the Hamilton-Jacobi equation in \eqref{eq1} determines the value function, $u(x)$, of a typical agent located at  $x$. The system   \eqref{eq1} gives  a stationary  Nash equilibrium that reflects the coupling between the evolution the location of the agents and their optimal actions. 

Various authors have investigated models related to   \eqref{eq1} both in the stationary and time-dependent setting. For example, the existence of weak and smooth solutions time-dependent models was investigated, respectively, in  \cite{cgbt},
\cite{porretta}, and \cite{porretta2} and in   \cite{GPim2},   \cite{GPim1},
\cite{GPM3}, and \cite{GPM2}, the existence of weak solutions for stationary
MFGs was considered in \cite{ll1}, \cite{ll2} and \cite{LCDF}, whereas classical
solutions for this problem were studied in  \cite{cirant},  \cite{GPatVrt},  \cite{GPM1},
and \cite{PV15}. 
However, relatively little is know for problems with congestion in spite of the fact that the existence and regularity of solutions of MFGs with congestion are of
fundamental interest. For these problems, the existence of classical solutions is established in \cite{GVrt2} for small terminal times. In \cite{Graber2}, a similar problem was studied with Neumann and Dirichlet conditions and proved the existence of weak solutions for small terminal times. The existence of a solution for a  stationary congestion problem with a quadratic Hamiltonian was established in \cite{GMit}. Moreover, in \cite{AL16} weak solutions for mean-field control problem with congestion were investigated. Finally, several explicit examples were examined in \cite{GLP2}, building upon the methods in \cite{GP}.  

In \cite{GMit}, a crucial a priori estimate for $m^{-1}$ was proven using a remarkable identity that relies on the structure of the Hamiltonian. This identity does not hold for general Hamiltonians. Here, to solve this issue, we develop new estimates and prove bounds for $m^{-1}$ in $L^p(\Tt^d)$ for any $p>1$.
In addition to these estimates, we prove estimates for congestion problems that, to the best of our knowledge are also new.  For example,  we consider the bound for $u\in L^\infty$ in Section \ref{ubu}.

Our main result is the following theorem:\begin{teo}\label{mtheo}
Suppose Assumptions \ref{A2}-\ref{A1}
 (see Section \ref{assp}) hold. Then, Problem \ref{P1} has a unique solution $(u,m)$, $u,m\in C^\infty(\Tt^d)$, $m>0$.  
\end{teo}

We end this introduction by outlining our paper. The main assumptions are
discussed in Section \ref{assp}. 
In Section \ref{Prelim_Est}, we present some preliminary estimates. Our most important contribution in the proof of Theorem \ref{mtheo} is a new estimate for $\frac 1 m$\ in $L^p$\ spaces. We discusse this estimate in  Section \ref{lower_bounds_m}.
Next, in Section  \ref{impreg}, we prove further regularity on $u$ and $m$. Finally, the proof of Theorem \ref{mtheo} is presented in Section 
\ref{existence_cont_meth}. 

\section{Main Assumptions}
\label{assp}

Next, we present the main assumptions for Problem \ref{P1}.
In contrast with 
\cite{GMit}, here, 
we assume a general form for the Hamiltonian. 
Our assumptions are natural and arise frequently in optimal control, viscosity solutions, and mean-field games.

\begin{hyp}\label{A2}
        The Hamiltonian satisfies:
\begin{equation*}
H(x,0)\leq 0 
\end{equation*}
for all $x\in \mathbb T^d$.
\end{hyp}
The preceding assumption gives a lower bound for the Lagrangian; we recall that the Lagrangian, $L$, 
is the Legendre transform of the Hamiltonian, $H$, and is determined by
\begin{equation}
\label{lag}
L(x,v)=\sup_{p} (-v \cdot p- H(x,p)). 
\end{equation}
Thus, $L(x,v)\geq -H(x,0)\geq 0$.
\begin{hyp}\label{boundV}
The potential $V:\Tt^d\times\Rr\to\Rr$ is globally bounded with bounded derivatives of all orders.
\end{hyp}

To simplify the presentation and to illustrate  the main difficulty in MFGs
with congestion --   the control of the term  $\frac 1 m$ -- we impose the
above boundedness conditions. However, the preceding assumption is far from optimal. With suitable modifications, our results  can be extended if $V$ has mild polynomial growth but 
some proofs have to be changed. For example, the proof of Proposition \ref{L1} as  presented here requires $V$ to be bounded. 
\begin{hyp}\label{A4}
There exist constants, $c,C>0$, such that
\begin{equation*}
D_pH(x,p)\cdot p - H(x,p)\geq cH(x,p)-C
\end{equation*}
for all $(x,p)\in \mathbb T^d\times\mathbb R^d$.
\end{hyp}
If the supremum in \eqref{lag} is achieved at a point $p\in \Rr^d$, we have $v=-D_pH(x,p)$. Therefore, 
\[
L(x,v)=D_pH(x,p)\cdot p-H(x,p).
\]
Accordingly, the preceding assumption gives the bound $L(x, v)\geq c H(x,p)-C$; that is, in classical mechanics language, the energy bounds  the action from below. 
\begin{hyp}\label{As1.2} 
There exist $\gamma>1$ and constants, $c,C>0$, such that $H$ satisfies
\[
c|p|^{\gamma}-C\leq H(x,p)\leq  c|p|^{\gamma}+C 
\]
for all $(x,p)\in \mathbb T^d\times\mathbb R^d$.

\end{hyp}

\begin{remark}\label{A3}
Note that Assumptions \ref{A4} and \ref{As1.2}  imply that there exist constants, $c,C>0$, such that
\begin{equation*}
D_pH(x,p)\cdot p - H(x,p)\geq c|p|^{\gamma}-C
\end{equation*}
for all $(x,p)\in \mathbb T^d\times\mathbb R^d$.
\end{remark}

\begin{remark}\label{R34}
By combining Assumption \ref{As1.2} with Remark \ref{A3}, we conclude that there exist constants, $c,C>0$, such that, for any $r>1$,
\[
c|p|^\gamma<H(x,p)+rp\cdot D_pH(x,p)+Cr
\]
for all $(x,p)\in \mathbb T^d\times\mathbb R^d$.
\end{remark}

\begin{hyp}\label{gamaSub}
$H$ has sub-quadratic growth; that is, $1<\gamma<2$.
\end{hyp}
        
Subquadratic growth allows the use of the Gagliardo-Nirenberg inequality. As previously mentioned, the quadratic case has been addressed previously by a particular identity  \cite{GMit}, and, in general, is still open.

\begin{hyp}\label{A6}
There exist a constant, $C>0$, such that 
\begin{equation*}
|D_pH(x,p)|\leq  C|p|^{\gamma-1}+C
\end{equation*}
for all $(x,p)\in \mathbb T^d\times\mathbb R^d$.
\end{hyp}
The growth condition in the previous assumption is natural in view of Assumption \ref{As1.2}.

\begin{hyp}\label{alphaexponent}
        The exponent $\alpha$ in the congestion term,  $m^\alpha$, satisfies \[
0<\alpha<2,\qquad  d<4+\frac{2}{\alpha}, \qquad 1<\gamma<1+\frac{1}{1+2\alpha}. 
\]

\end{hyp}

In the preceding assumption, we chose to present a single condition under which all our results hold. However, some results require less restrictive assumptions. The bounds for $u$\ in $L^\infty$ in Proposition \ref{bounduabove}  are valid  if $d\leq 4$ for any $\alpha\geq 0$ and for $\alpha<\frac{2}{d-4}$ if $d>4$. Corollary
\ref{1mInLr} requires the additional condition $\alpha<\frac{2-\gamma}{2(\gamma-1)}$ to get $\frac 1 m\in L^p$ for any $1<p<\infty$. Here, we see that the range of the congestion exponent, $\alpha$,  collapses  as $\gamma\to
2$. Proposition \ref{L9} and Corollary \ref{1mInLr} are two of the three critical steps in our proof in which the subquadratic behavior of the Hamiltonian is crucial. The other is the use of the Gagliardo-Nirenberg inequality in Proposition \ref{DuLgamma}.
Moreover, Proposition \ref{DuLgamma} requires an additional constraint on $\alpha,$
\[
\frac{2 (\alpha+1) }{(\gamma-1) \alpha (d-2)}>1.
\] 
Finally, as shown in \cite{Gueant2}, MFGs with congestion have unique solutions in  a suitable range of values of $\alpha$.
To apply the Implicit Function Theorem in the proof of Theorem \ref{mtheo}, we need an additional constraint on $\alpha$: $0<\alpha<2$. The conditions in the preceding Assumption result from combining the aforementioned constraints.

The next two assumptions are monotonicity assumptions for the MFG. In this paper, these are used  to ensure uniqueness and to prove
existence using the continuation method in Section \ref{existence_cont_meth}.
In the existence theory,  the strict monotonicity can be relaxed  by a limiting argument
since because of our a priori bounds depends on strict monotonicity.

\begin{hyp}\label{Vincreasing}
        The potential $V:\Tt^d\times\Rr\to \Rr$ is strictly decreasing in the second variable.       
\end{hyp}

\begin{hyp}\label{A1}
For every $x\in \Tt^d$ and $p\in \Rr^d$, we have
$D^2_{pp}H(x,p)> 0$ and  \[
D_pH(x,p)\cdot p -H
(x,p)>\frac\alpha 4p^T\cdot D_{pp}^2H (x,p)\cdot p.
\]
\end{hyp}
To end this section, we prove that example \eqref{eq:example} satisfies assumptions \ref{A2}, \ref{A4}, \ref{As1.2}, \ref{A6} and \ref{A1}, for $\gamma$ in the range of Assumption \ref{gamaSub}, and $\frac 1 \gamma+\frac{1}{\gamma'}=1$. Assumption \ref{alphaexponent} gives further restrictions for $\gamma$ and the remaining assumptions concern $V$.

First, Assumption \ref{A2} holds because $a>0$. Regarding Assumption \ref{A4}, we have to show that there exist constants, $c>0$, $C>0$, such that
        \[
        L(x,v)=a(x)(1+|v|^2)^{\frac{\gamma'}{2}}\geq cH(x,p)-C.
        \]
        By definition, we have
        \begin{equation}\label{supwithLv}
        H(x,p)=\sup_{v\in\Vv}\{-p\cdot v-a(x)(1+|v|^2)^{\frac{\gamma'}{2}}\}.
        \end{equation}
The supremum in the above expression is achieved at $v$ if
        \begin{equation}\label{eqforpopt}
        p=-\gamma' v a(x)(1+|v|^2)^{\frac{\gamma'}{2}-1}.
        \end{equation}
        This implies that for $p$ as above
        \begin{align*}
        H(x,p)=a(x)((\gamma'-1)|v|^2-1)(1+|v|^2)^{\frac{\gamma'}{2}-1}\leq (\gamma'-1)L(x,v).
        \end{align*}
        Thus, Assumption \ref{A4} holds.
       Next, note that, from \eqref{eqforpopt}, there exist constants, $c, C>0,$ such that
       \begin{equation}\label{ineqforv}
       c|p|^{\gamma-1}\leq |v|\leq {C}|p|^{\gamma-1}.
       \end{equation}
       Because $1<\gamma<2$, we have $\gamma'>2$. Therefore, we have the estimate 
       \begin{align}\label{ineq1H}
       \begin{split}
       H(x,p)&=a(x)((\gamma'-1)|v|^2-1)(1+|v|^2)^{\frac{\gamma'}{2}-1}\\
       &\leq a(x)(1+(\gamma'-1)|v|^2)^{\frac{\gamma'}{2}}\\
       &\leq C |v|^{\gamma'}\\
       &\leq C|p|^\gamma +C
       \end{split}
       \end{align}
       for some $C>0$. On the other hand, we have
         \begin{align}\label{ineq2H}
         \begin{split}
         H(x,p)&=a(x)((\gamma'-1)(1+|v^2|)-\gamma')(1+|v|^2)^{\frac{\gamma'}{2}-1}\\
         &=a(x)(\gamma'-1)(1+|v|^2)^{\frac{\gamma'}{2}}-\gamma' a(x)(1+|v|^2)^{\frac{\gamma'}{2}-1}\\
         &\geq c|v|^{\gamma'} - C|v|^{\gamma'-1}-C\\
         &\geq c|v|^{\gamma'} -C.
         \end{split}
         \end{align}
       The last line follows from Young's inequality. Thus, from \eqref{ineq1H} and \eqref{ineq2H}, $H$ satisfies Assumption \ref{As1.2}.
           Observe that the optimality condition for \eqref{supwithLv} gives
\begin{equation}\label{vopt}
v=-D_pH(x,p). 
\end{equation}
Consequently, \eqref{ineqforv} implies Assumption \ref{A6}.  

Moreover, differentiating \eqref{vopt} with respect to $p$, we have
\[
D_pv=-D_{pp}^2H(x,p).
\]
Since $p^{}=-D_vL(x,v)$ and  differentiating with respect to $p$ again, we have
\begin{align*}
\begin{split}
I&=-D_{vv}^2L(x,v) D_pv\\
&=D_{vv}^2L(x,v) D_{pp}^2H(x,p). 
\end{split}
\end{align*}
Now, note that 
\begin{align}\label{DvDvv}
\begin{split}
D_{vv}L^2(x,v)&=\gamma'a(x)(1+|v|^2)^{\frac{\gamma'}{2}-1}\left(I+(\gamma'-2)\frac{v\otimes v}{1+|v|^2}\right).
\end{split}
\end{align}
Therefore, $D_{pp}H(x,p)=(D_{vv}^2L(x,v))^{-1}>0$. To verify the other condition in Assumption \ref{A1}, we need to prove that
\begin{equation}\label{ineqA1}
L(x,v)>\frac{\alpha}{4}p^T  D_{pp}H(x,p)  p.
\end{equation}
Note that $v$ is an eigenvector for $D_{vv}^2L(x,v)$:
\[
D_{vv}L^2(x,v)  v=\gamma'a(x)(1+|v|^2)^{\frac{\gamma'}{2}-1}\left(1+(\gamma'-2)\frac{|v|^2}{1+|v|^2}\right) v.
\]
Therefore, using \eqref{DvDvv}, we obtain
\[
v^T D_{pp}^2H(x,p) v=\frac{|v|^2}{\gamma'a(x)(1+|v|^2)^{\frac{\gamma'}{2}-1}\left(1+(\gamma'-2)\frac{|v|^2}{1+|v|^2}\right)}.
\]
From the expression for $p$ in \eqref{eqforpopt}, we have

\begin{align*}
\frac{\alpha}{4}p^T D_{pp}^2H(x,p) p &=\frac{\alpha}{4}\frac{|v|^2(\gamma')^2a(x)^2[(1+|v|^2)^{\frac{\gamma'}{2}-1}]^2}{\gamma'a(x)(1+|v|^2)^{\frac{\gamma'}{2}-1}\left(1+(\gamma'-2)\frac{|v|^2}{1+|v|^2}\right)}\\
&=\frac{\alpha}{4}\frac{|v|^2\gamma'a(x)(1+|v|^2)^{\frac{\gamma'}{2}-1}}{\left(1+(\gamma'-2)\frac{|v|^2}{1+|v|^2}\right)}.
\end{align*}
Therefore, \eqref{ineqA1} can be rewritten as
\begin{align*}
\frac{\alpha}{4}\frac{|v|^2\gamma'a(x)(1+|v|^2)^{\frac{\gamma'}{2}-1}}{\left(1+(\gamma'-2)\frac{|v|^2}{1+|v|^2}\right)}< a(x)(1+|v|^2)^{\frac{\gamma'}{2}}.
\end{align*}
The prior inequality is equivalent to
\begin{align*}
{\alpha}&< 4\left(1+(\gamma'-2)\frac{|v|^2}{1+|v|^2}\right)\frac{1+|v|^2}{\gamma'|v|^2}\\
&=4\frac{(1+(\gamma'-1)|v|^2)}{\gamma'|v|^2}\\
&=4\left(\frac{1}{\gamma'|v|^2}+\frac{1}{\gamma}\right)=:\tilde{\alpha}.
\end{align*}
Because $\gamma',|v|^2\geq 0$, we have $\tilde{\alpha}>\frac{4}{\gamma}>2$. Hence, the example satisfies Assumption \ref{A1} and $\alpha$ can be chosen in the range given by Assumption \ref{alphaexponent}.
This argument shows that the example satisfies our assumptions.

\section{Preliminary estimates}\label{Prelim_Est}
We begin  the study of Problem \ref{P1} by establishing several a priori estimates. First, we use the maximum principle to get lower bounds for $u$. Next, we prove energy estimates using a method introduced in \cite{ll1} and obtain bounds for $m$ in 
$L^{(1+\alpha)2^*/2}(\Tt^d)$.
Subsequently, we use the nonlinear adjoint method to get upper bounds for $u$.
As a consequence, we obtain a priori bounds for $u$ in $L^\infty(\Tt^d)$ and  in  $W^{1,\gamma}(\Tt^d)$. 
  
\subsection{Maximum principle}
Here, we show that $u$ is bounded from below and that $u\in L^1(\Tt^d)$. For that, we use the maximum principle to get lower bounds on $u$\ and an elementary integration argument to get $L^1$ bounds.  
\begin{pro}\label{L1}
 Suppose that Assumptions \ref{A2} and \ref{boundV} hold. Then, there exists a constant,
$C>0$, such that, for any solution, $(u,m),$ of  Problem \ref{P1}, we have $u\geq -C$. 
Furthermore,
$\|m\|_{L^1(\mathbb T^d)}=1$.   
\end{pro}
\begin{proof}
To get the lower bound, we evaluate the equation at a point
of minimum of $u$. At a minimum, $Du=0$ and $\Delta u\geq 0$.
Because Assumption \ref{A2} gives $H(x,0)\leq 0$ for all $x\in \Tt^d$, the bound follows using Assumption \ref{boundV}.  Because $m\geq 0$, by integrating the second equation in \eqref{eq1}, we find that $m$ has a total mass of 1 and, consequently, $\|m\|_{L^1(\mathbb T^d)}=1$.    
\end{proof}

\begin{corollary}\label{corolboundu}
Suppose that Assumptions \ref{A2} and \ref{boundV} hold. Then, there exists a
constant,
$C>0$, such that, for any solution, $(u,m),$ of  Problem \ref{P1}, we have 
\begin{equation*}
 \int_{\mathbb T^d} |u| dx \leq C.
\end{equation*}
\end{corollary}
\begin{proof}
Integrating the first equation in \eqref{eq1} gives
\[
\int_{\Tt^d}udx\leq C.
\]      
The result follows by using the lower bound on $u$ from Proposition \ref{L1}.
\end{proof}
\subsection{First-order estimates}
Now, we develop first-order estimates for $u$ and $m$. First, Proposition \ref{propFO1} and Corollary \ref{corFO1}
provide bounds for $Du$\ that depend on $\int_{\Tt^d} m^{1+\alpha} dx$. Then, to close these bounds, we show   in Proposition \ref{L9} 
that $m\in L^{(1+\alpha)2^*/2} (\Tt^d)$, where $2^*=\frac{2d}{d-2}$ is the Sobolev conjugate exponent of $2$.

\begin{pro}\label{propFO1}  
        
        Suppose that Assumptions  \ref{A2}-\ref{A4} hold. Then, there exists a
constant,
$C>0$, such that, for any solution, $(u,m),$ of  Problem \ref{P1}, we have

\begin{equation}\label{boundHdepm}
\int_{\Tt^d}\left[m^\alpha H\left(x,\frac{Du}{m^\alpha}\right)+m^{1+\alpha}H\left(x,\frac{Du}{m^\alpha}\right)\right]dx\leq
C\int_{\Tt^d}m^{1+\alpha}dx+C.
\end{equation}
\end{pro}
\begin{proof}
We add the first equation in \eqref{eq1} multiplied by $(m-1)$ to the second equation multiplied by $-u$. Then, integration by parts yields
\[
\int_{\Tt^d}m^{1+\alpha}\left[ H\left(x,\frac{Du}{m^\alpha}\right)-D_pH\left(x,\frac{Du}{m^\alpha}\right)\cdot\frac{Du}{m^\alpha} \right]dx=\int_{\Tt^d}\left[ m^\alpha H\left(x,\frac{Du}{m^\alpha}\right)+(1-m)V \right]dx.
\]
Finally, combining Proposition \ref{L1}, the boundedness of $V$ from Assumption \ref{boundV} with Assumption \ref{A4} gives \eqref{boundHdepm}.
\end{proof}

\begin{corollary}\label{corFO1}
        Suppose Assumptions \ref{A2}--\ref{As1.2} hold. Then, there exists a
constant,
$C>0$, such that, for any solution, $(u,m),$ of  Problem \ref{P1}, we have

\begin{equation*}
\int_{\Tt^d} |Du|^\gamma m^{\beta} dx\leq \int_{\Tt^d}m^{1+\alpha}dx+C
\end{equation*}         
for all $-\bar \alpha \leq\beta\leq1-\bar \alpha $,
where
\begin{equation}\label{defalphabar}
\bar{\alpha}=(\gamma-1)\alpha.
\end{equation}
\end{corollary}
\begin{proof}
Assumption \ref{As1.2} gives
\[
c\int_{\Tt^d} \frac{|Du|^\gamma}{m^{\bar\alpha}} dx\leq \int_{\Tt^d}H\left(x,\frac{Du}{m^\alpha}\right)m^\alpha dx+C\int_{\Tt^d}m^{\alpha}dx
\]
and
\[
c\int_{\Tt^d} \frac{|Du|^\gamma}{m^{\bar\alpha-1}} dx\leq \int_{\Tt^d}H\left(x,\frac{Du}{m^{\alpha}}\right)m^{1+\alpha} dx+C\int_{\Tt^d}m^{1+\alpha}dx.
\]
Because $m^{\alpha}\leq Cm^{1+\alpha} +C$, the result follows from Proposition \ref{propFO1} combined with the inequality $m^{\beta}\leq m^{-\bar\alpha}+m^{1-\bar\alpha}$ for all $-\bar \alpha \leq{\beta}\leq1-\bar \alpha $.
\end{proof}

\begin{pro}\label{L9}
Suppose Assumptions \ref{A2}-\ref{A6} hold. Then, there
exists a
constant,
$C>0$, such that, for any solution, $(u,m),$ of  Problem \ref{P1},
we have
\begin{equation*}
\int_{\Tt^d} m^{1+\alpha}dx+\int_{\Tt^d} |Dm^{(1+\alpha)/2}|^2dx\leq C.
\end{equation*}
Moreover, we have
\begin{equation*}
\|m\|_{L^{(1+\alpha)2^*/2}(\Tt^d)}\leq C.
\end{equation*}
\end{pro}

\begin{proof}
First, we multiply the second equation in \eqref{eq1} by $m^{r}$ for some $r>0$ to be chosen later. Next, we integrate
by parts to get       
\begin{align*}
\int_{\Tt^d} m^{1+r}dx+r\int_{\Tt^d}m^{r-1}|Dm|^2 dx
+r\int_{\Tt^d}m^{r} Dm\cdot D_pH\left(x,\frac{Du}{m^\alpha}\right) dx= \int_{\Tt^d} m^{r}dx. 
\end{align*}
Then, Young's and Cauchy's inequalities imply
        
\begin{equation}\label{eqnoassu}
\int_{\Tt^d} m^{1+r}dx+\frac{r}{2}\int_{\Tt^d} m^{r-1}|Dm|^2dx\leq\frac{r}{2}\int_{\Tt^d}m^{1+r}\left|D_pH\left(x,\frac{Du}{m^\alpha}\right)\right|^2
dx+C\int_{\Tt}m^{1+r}dx+C,    
\end{equation}
where $C$ depends on $r$. Furthermore, Assumption \ref{A6}  yields

\begin{align*}
\frac{r}{2}\int_{\Tt^d}m^{1+r}&\left|D_pH\left(x,\frac{Du}{m^\alpha}\right)\right|^2
dx+C\int_{\Tt}m^{1+r}dx+C\\
&\leq\int_{\Tt^d}\left[ m^{1+r}\left(\frac{|Du|}{m^{\alpha}}\right)^{2(\gamma-1)} +Cm^{1+r}\right]dx+C\int_{\Tt}m^{1+r}dx+C\\
&\leq\int_{\Tt^d} [C|Du|^{2(\gamma-1)}m^{1+r-2\bar{\alpha}}+Cm^{1+r}]dx+C.\\
\end{align*}
Now, for $s,t>1,$ we write
\[
|Du|^{2(\gamma-1)}m^{1+r-2\bar{\alpha}}=(|Du|^\gamma m^{1-\bar{\alpha}})^{\frac 1 s}(m^{1+r})^{\frac 1 t},
\]
where 
\begin{equation*}
\begin{cases}
\quad2(\gamma-1)&=\frac{\gamma}{s}\\
1+r-2\bar{\alpha}&=\frac{1-\bar{\alpha}}{s}+\frac{1+r}{t}\\
\quad\quad\frac 1 s + \frac 1 t&=1.
\end{cases}
\end{equation*}
Solving the previous system for $r,s,$ and $t$, we find
\[
r=\frac{\bar\alpha}{\gamma-1}=\alpha,\quad s=\frac{\gamma}{2(\gamma-1)},\quad t=\frac{\gamma}{2-\gamma},
\]
where we used the definition of  $\bar{\alpha}$ in \eqref{defalphabar}.
Due to Assumption \ref{gamaSub}, $s,t>1$. Thus, using the above expressions for $r$, $s,$ and $t$ and  Young's inequality, we get
\begin{align*}
|Du|^{2(\gamma-1)}m^{1+\alpha-2\bar{\alpha}}&=(|Du|^\gamma m^{1-\bar{\alpha}})^{\frac{2(\gamma-1)}{\gamma}}(m^{1+\alpha})^{\frac{2-\gamma}{\gamma}}\quad \\
&\leq \frac{2(\gamma-1)}{\gamma}|Du|^\gamma m^{1-\bar{\alpha}}
+\frac{2-\gamma}{\gamma}m^{1+\alpha}.\\
\end{align*}
Note that 
\(m^{\xi-1}|Dm|^2=\frac{4}{(\xi+1)^2}|Dm^{(\xi+1)/2}|^2 \)
for all $\xi\neq -1$.
Therefore, with the previous estimates and Corollary \ref{corFO1}, \eqref{eqnoassu} implies that \begin{equation*}
\int_{\Tt^d} m^{1+\alpha}dx+\frac{1}{2}\int_{\Tt^d} |Dm^{(1+\alpha)/2}|^2dx\leq C\int_{\Tt^d}m^{1+\alpha}
dx+C,
\end{equation*}
where $C$ depends on $\alpha$.

Consequently, using $\int_{\Tt^d} m dx=1$, the Sobolev inequality, and H\"older's inequality, 
we conclude that
\begin{equation*}
\int_{\Tt^d} m^{1+\alpha}dx+\int_{\Tt^d} |Dm^{(1+\alpha)/2}|^2dx\leq C.
\end{equation*}
\end{proof}
Thanks to the preceding result, we improve Proposition \ref{propFO1} and Corollary \ref{corFO1} as follows.
\begin{corollary}\label{corboundDum}
Suppose that Assumptions \ref{A2}-\ref{A6} hold. Then, there
exists a
constant,
$C>0$, such that, for any solution, $(u,m),$ of  Problem \ref{P1},
we have
    
\begin{equation}\label{Hmbound}
\int_{\Tt^d}\left[m^\alpha H\left(x,\frac{Du}{m^\alpha}\right)+m^{1+\alpha}H\left(x,\frac{Du}{m^\alpha}\right)\right]dx\leq C
\end{equation}
and 
\begin{equation}\label{Dumbound}
\int_{\Tt^d} |Du|^\gamma m^{\beta} dx\leq C
\end{equation}
for all $-\bar \alpha \leq\beta\leq 1-\bar \alpha$.
\end{corollary}

\begin{proof}
To obtain \eqref{Hmbound} and \eqref{Dumbound}, we combine Proposition \ref{propFO1}, Corollary \ref{corFO1}, and Proposition \ref{L9}.
\end{proof}

\subsection{Estimates for the Fokker-Planck equation}

Now, we examine the Fokker-Planck equation and prove an elementary entropy
bound. 

\begin{pro}\label{L4} 
Suppose that Assumptions \ref{A2}-\ref{A6} hold. Then, there
exists a
constant,
$C>0$, such that, for any solution, $(u,m),$ of  Problem \ref{P1},
we have

\begin{align*}
\int_{\mathbb T^d} m \log m dx + \int_{\mathbb T^d} |Dm^{1/2}|^2 dx\leq C.
\end{align*}

\end{pro}
\begin{proof}
We multiply the second equation in \eqref{eq1} by $\log m$ and integrate
by parts to get
\begin{equation*}
\int_{\mathbb T^d} m\log m dx+\int_{\mathbb T^d} \frac{|Dm|^2}{m}
dx+\int_{\mathbb T^d}Dm \cdot D_pH\left(x,\frac{Du}{m^\alpha}\right) dx=\int_{\mathbb
T^d}\log m dx.
\end{equation*}
Next, we use Cauchy inequality and Jensen's inequality for $\log m$ to obtain

\begin{align*}
\int_{\mathbb T^d} m \log m dx+ \int_{\mathbb T^d}\frac{|Dm|^2}{m} dx&\leq
-\int_{\mathbb T^d} Dm \cdot D_pH\left(x,\frac{Du}{m^\alpha}\right) dx\\
&\leq \frac 1 2 \int_{\mathbb T^d} \frac{|Dm|^2}{m} dx + C\int_{\mathbb
T^d}\left|D_pH\left(x,\frac{Du}{m^\alpha}\right)\right|^2m dx.
\end{align*}
Finally, according to Assumption \ref{A6} and Corollary \ref{corboundDum},
we have
\begin{align*}
\int_{\mathbb T^d} m \log m dx + \int_{\mathbb T^d} |Dm^{1/2}|^2 dx&\leq
C\int_{\mathbb T^d}\left|D_pH\left(x,\frac{Du}{m^\alpha}\right)\right|^2m
dx\\
&\leq C\int_{\mathbb T^d}|Du|^{2(\gamma-1)}m^{1-2\bar{\alpha}} dx+C\\
&\leq C\int_{\mathbb T^d}|Du|^{\gamma}m^{1-2\bar{\alpha}} dx+C\leq C,\\
\end{align*}
where we used  $2 (\gamma-1)<\gamma$ from Assumption \ref{gamaSub}.
\end{proof}

\subsection{An upper bound on $u$}
\label{ubu}
Here, we use the nonlinear adjoint method to get an upper bound for $u$. Thus, we prove that using the lower bound for $u$ in Proposition \ref{L1}, we have $u\in L^\infty(\Tt^d)$.
Furthermore, using \eqref{Dumbound} with $\beta=0$ and Proposition \ref{bounduabove}, we get $u\in W^{1,\gamma}(\Tt^d)$ in Corollary \ref{uW1gamma}. 

\begin{pro}\label{bounduabove}
Suppose Assumptions \ref{A2}--\ref{alphaexponent} hold. Then, there
exists a
constant,
$C>0$, such that, for any solution, $(u,m),$ of  Problem \ref{P1},
we have
$u\leq C$.
\end{pro}
\begin{proof}
Let $0\leq \tau \leq T$ be fixed. Because $u$ and $m$ are functions on $\Tt^d$, we regard $u$ and $m$ as time-independent functions in $\Tt^d\times [\tau,T]$. Hence,  $u,m\in C^\infty(\Tt^d\times [\tau,T])$, with $u_t=m_t=0$, $u(x,T)=u(x)$ and $m(x,\tau)=m(x)$. We add $-u_t$ and $m_t$ to the first and second equation of \eqref{eq1}, respectively, to obtain
\begin{equation}\label{eq.ut}
\begin{cases}
&-u_t+u-\Delta u +m^{\alpha}H\left(x,\frac{Du}{m^{\alpha}}\right) + V(x,m) =  0\qquad \mbox{ in } \Tt^d\times [\tau,T],\\
& m_t+m -\Delta m - \div {\left( D_pH\left( x,\frac{Du}{m^{\alpha}}\right)m\right)}
=1\hspace{1.1cm} \mbox{ in } \Tt^d\times [\tau,T],\\ 
&u(x,T)=u(x),\; m(x,0)=m(x).

\end{cases}
\end{equation}
Let $u$ and $m$ solve Problem \ref{P1}. By construction and from the initial-terminal conditions in \eqref{eq.ut}, we have $u(x,t)=u(x)$ and $m(x,t)=m(x)$ for all $x\in\Tt^d\times [\tau,T]$. Let $\rho$ solve
\begin{equation}\label{eq.rho}
\begin{cases}
\rho_t -\Delta \rho +\rho &=0\quad \mbox{in } \Tt^d\times [\tau,T]\\
\qquad\qquad\quad \rho &= \delta_{x_\tau} \quad \mbox{on } \Tt^d\times \{t=\tau\}
\end{cases}
\end{equation}
for some $x_\tau\in\Tt^d$. We subtract the first equation of \eqref{eq.rho} multiplied by $u$ from the first equation in \eqref{eq.ut} multiplied by $\rho$ and integrate by parts to deduce that\[
\frac{d}{dt} \int_{\Tt^d} \rho u dx  = \int_{\Tt^d}\rho\left(m^\alpha H\left(x,\frac{Du}{m^\alpha}\right) +V \right)dx.
\]
Integrating from $\tau$ to $T$ and using Assumption \ref{As1.2}, we have,
\begin{align*}
u(x_\tau)=
\int_{\Tt^d}u(x)\delta_{x_\tau} dx&\leq C\int_\tau^T \int_{\Tt^d}\rho(x,t) m(x)^\alpha dx dt-\int_\tau^T \int_{\Tt^d}\rho(x,t) V(x,m(x)) dx dt \\ &\quad+\int_{\Tt^d}u(x)\rho(x,T)dx.
\end{align*}
Because $\rho$ solves \eqref{eq.rho}, we have $\rho(x,t) = e^{-t} \eta(x,t)$, where $\eta$ is the fundamental solution to 
\begin{equation*}
\begin{cases}
\eta_t -\Delta \eta&=0\quad \mbox{in } \Tt^d\times (\tau,\infty)\\
\hspace{1cm}\eta &= \delta_{x_\tau} \quad \mbox{on } \Tt^d\times \{t=\tau\}.
\end{cases}
\end{equation*}
From the previous estimate,  Holder's inequality and recalling that $V(x,m(x))$ does not depend on $t$, we deduce
that\begin{align*}
u(x_\tau)&\leq  C\int_{\tau}^{T}e^{-t}\|m^\alpha\|_{L^{\frac{2^*(\alpha+1)}{2\alpha}}(\Tt^d)}\|\eta(\cdot,t)\|_{L^{q}(\Tt^d)}dt\\
&\quad + \int_\tau^T e^{-t}\|\eta(\cdot,t)\|_{L^1(\Tt^d)}\|V\|_{L^{\infty}(\Tt^d)}dt\\
&\quad+e^{-T}\|u\|_{L^1(\Tt^d)}\|\eta(\cdot,T)\|_{ L^{\infty}(\Tt^d)},
\end{align*}
where $\frac 1 q+\frac{2 \alpha}{2^*(\alpha+1)}=1$. We recall, see \cite[Section 3, Lemma 1]{GVrt2}, that if  $1<q<\frac{2^*}{2}$, we have
\[
\|\eta\|_{L^1(L^q(dx),dt)}=\int_{\tau}^T\left(\int_{\Tt^d}\eta^q dx\right)^{\frac{1}{q}}dt\leq C_{q}.
\]
 
This condition on $q$ is ensured by  Assumption \ref{alphaexponent}.
Next, using Corollary \ref{corolboundu}, Proposition \ref{L9}, $\|\eta(\cdot,t)\|_{L^1(\Tt^d)}\leq 1$ for all $t\geq 0$, and  
Holder's inequality, we obtain,
\begin{align*}
u(x_\tau)&\leq C 
 + (e^{-\tau}-e^{-T})\|V\|_{L^{\infty}(\Tt^d)} +C e^{-T}.
\end{align*}
Because $\tau $ and $T$ are arbitrary and Assumption \ref{boundV} holds, we have $u\leq C$, where $C$ depends on $\alpha$ and $\|V\|_{L^\infty(\Tt^d)}$.
\end{proof}

\begin{cor}\label{uLinfty}
Suppose that Assumptions \ref{A2}-\ref{alphaexponent} hold. Then, there
exists a
constant,
$C>0$, such that, for any solution, $(u,m),$ of  Problem \ref{P1},
we have\[
\|u\|_{L^\infty(\Tt^d)}\leq C.
\]
\end{cor}
\begin{proof}
The result follows by combining Propositions \ref{bounduabove} and \ref{L1}.
\end{proof}

\begin{cor}\label{uW1gamma}
Suppose that Assumptions \ref{A2}-\ref{alphaexponent} hold. Then, there
exists a
constant,
$C>0$, such that, for any solution, $(u,m),$ of  Problem \ref{P1},
we have $\|u\|_{ W^{1,\gamma}(\Tt^d)}\leq C$.
\end{cor}
\begin{proof}
Note that Corollary \ref{corboundDum} gives $Du\in L^{\gamma}(\Tt^d)$. Then, the result follows from Corollary \ref{uLinfty}.
\end{proof}

\section{Lower bounds on $m$}\label{lower_bounds_m}

Because of the singularity that arises in Problem \ref{P1} when $m=0$, it is essential to get lower bounds for $m$. Thus, here, we investigate $L^p$ bounds for $\frac 1 m$. 

\begin{pro}\label{firstlbPro}
        Suppose that Assumptions \ref{A2}-\ref{alphaexponent} hold. Then, for any $r>1$, there
exists a constant, $C_r$, such that, for any solution, $(u,m)$, of  Problem \ref{P1},
we have
        \begin{align}
\int_{\Tt^d}\frac{1}{m^{r+1}}dx+
\int_{\Tt^d}\left|D\frac{1}{m^{r/2}}\right|^2dx+\int_{\Tt^d}\frac{|Du|^\gamma}{m^{r+\bar\alpha}}dx \leq
C_r\int_{\Tt^d}\frac{1}{m^{r+\delta}}dx+C_r,\nonumber
\end{align}
where
        \begin{equation*}
        \delta=\frac{2\bar{\alpha}}{2-\gamma}.
        \end{equation*}
\end{pro}
\begin{proof}
        According to Corollary \ref{uLinfty}, $\Lambda=\|u\|_{L^\infty}<C$ for some constant that depends only on the problem data. We write the first equation of \eqref{eq1} as        \begin{equation}\label{eqlambdau}
        (u-\Lambda -1)-\Delta u+m^\alpha H\left(x,\frac{Du}{m^\alpha}\right)+V+\Lambda+1=0.
        \end{equation}
        We fix $r>1$ and use the notation $H=H\left(x,\frac{Du}{m^\alpha}\right)$.
Next, we multiply \eqref{eqlambdau} by $\frac{1}{m^r}$ and add it to the second
equation of \eqref{eq1} multiplied by $r\frac{(u-\Lambda-1)}{m^{r+1}}$. After
integrating by parts, we have,
        \begin{align*}
        &(r+1)\int_{\Tt^d}\frac{(u-\Lambda-1)}{m^r}dx+\int_{\Tt^d}m^\alpha
\frac{[H+r\frac{Du}{m^\alpha}\cdot D_pH]}{m^r}dx-r(r+1)\int_{\Tt^d}\frac{(u-\Lambda-1)|Dm|^2}{m^{r+2}}dx\\
        &-r(r+1)\int_{\Tt^d}\frac{(u-\Lambda-1)D_pH\cdot Dm}{m^{r+1}}dx-r\int_{\Tt^d}\frac{(u-\Lambda-1)}{m^{r+1}}dx+\int_{\Tt^d}\frac{V+\Lambda+1}{m^r}=0.
        \end{align*}
        Therefore, we get
        \begin{align*}
        &-(r+1)(2\Lambda+1)\int_{\Tt^d}\frac{1}{m^r}dx+\int_{\Tt^d}m^\alpha
\frac{[H+r\frac{Du}{m^\alpha}\cdot D_pH]}{m^r}dx+\frac{4(r+1)}{r}\int_{\Tt^d}\left|D\frac{1}{m^{r/2}}\right|^2dx\\
        &-r(r+1)(2\Lambda +1)\int_{\Tt^d}\frac{|D_pH| |Dm|}{m^{r+1}}dx+r\int_{\Tt^d}\frac{1}{m^{r+1}}dx+\int_{\Tt^d}\frac{V+\Lambda+1}{m^r}\leq 0.
        \end{align*}
From Remark \ref{R34}, there exist constants, $c,C>0$,  such that,
        \[
        c\frac{|Du|^\gamma}{m^{r+\bar{\alpha}}}\leq m^\alpha\frac{[H+r\frac{Du}{m^\alpha}\cdot
D_pH]}{m^r}+\frac{C r}{m^{r-\bar{\alpha}}}.
        \]
        We combine the preceding estimates to obtain           \begin{align*}
        &r\int_{\Tt^d}\frac{1}{m^{r+1}}dx+c\int_{\Tt^d}\frac{|Du|^\gamma}{m^{r+\bar\alpha}}dx+\frac{4(r+1)}{r}\int_{\Tt^d}\left|D\frac{1}{m^{r/2}}\right|^2dx\leq\\
&Cr(r+1)\int_{\Tt^d}\frac{|D_pH||Dm|}{m^{r+1}}dx
        +C r\int_{\Tt^d}\frac{1}{m^r}dx +rC\int_{\Tt^d}\frac{1}{m^{r-\alpha}}dx, 
        \end{align*}
       where the constants on the right-hand side depend only on the problem data. Dividing both sides by $r$, we have

      \begin{align*}
      &\int_{\Tt^d}\frac{1}{m^{r+1}}dx+\frac{c}{r}\int_{\Tt^d}\frac{|Du|^\gamma}{m^{r+\bar\alpha}}dx+\frac{4(r+1)}{r^2}\int_{\Tt^d}\left|D\frac{1}{m^{r/2}}\right|^2dx\leq
      C(r+1)\int_{\Tt^d}\frac{|D_pH||Dm|}{m^{r+1}}dx\\
      &\hspace{2cm}+C \int_{\Tt^d}\frac{1}{m^r}dx +C\int_{\Tt^d}\frac{1}{m^{r-\alpha}}dx.
      \end{align*}
       Next, Young's inequality gives,
        \begin{align*}
        \frac{C}{m^{r}}&\leq \frac{1}{4m^{r+1}}+C_r^1
        \end{align*}
        and
        \begin{align*}
        \frac{C}{m^{r-\alpha}}&\leq 
        \frac{1}{4m^{r+1}}+C_r^1
        \end{align*}
        
        for a suitable constant, $C_r^1$. Using the above inequalities and Assumption \ref{A6}, we find that
         \begin{align}\label{fulleq1}
         &\frac{1}{2}\int_{\Tt^d}\frac{1}{m^{r+1}}dx+\frac{c}{r}\int_{\Tt^d}\frac{|Du|^\gamma}{m^{r+\bar\alpha}}dx+\frac{r}{r+1}\int_{\Tt^d}\left|D\frac{1}{m^{r/2}}\right|^2dx\\&\hspace{4cm}\leq
         C(r+1)\int_{\Tt^d}|Du|^{\gamma-1}\frac{|Dm|}{m^{r+\bar{\alpha}+1}}dx+C_r^2\nonumber
         \end{align}
         for some constant, $C^2_r$.  For $p_1,p_2,p_3>1$, we write
\[
\frac{|Du|^{\gamma-1}|Dm|}{m^{r+\bar\alpha+1}} = \left(\frac{|Du|^\gamma}{m^{r+\bar{\alpha}}}\right)^{\frac{1}{p_1}}\left(\frac{|Dm|^2}{m^{r+2}}\right)^{\frac{1}{p_2}}\left(\frac{1}{m^{r+\delta}}\right)^{\frac{1}{p_3}},
\]   
where
\[
\begin{cases}
&\frac{1}{p_1}+\frac{1}{p_2}+\frac{1}{p_3}=1\\
&\gamma-1=\frac{\gamma}{p_1}\\
&\frac{2}{p_2}=1\\
&r+\bar{\alpha}+1=\frac{r+\bar{\alpha}}{p_1}+\frac{r+2}{p_2}+\frac{r+\delta}{p_3}.
\end{cases}
\]      
Solving the previous system for $p_1,p_2,p_3$ and $\delta$, we find
that\[
p_1=\frac{\gamma}{\gamma-1}, \quad p_2=2, \quad p_3=\frac{2\gamma}{2-\gamma}, \quad \delta=\frac{2\bar{\alpha}}{2-\gamma}.
\] 

Because Assumption \ref{gamaSub} holds, $p_1, p_2, p_3>1$. We use Young's inequality with $\epsilon_1=\frac{(cr-1)\gamma}{C(\gamma-1)r^2(r+1)}$ and $\epsilon_2=\frac{r^2 \left(r^2-r-1\right)}{2C(r+1)^2}$. Note that $\epsilon_1>0$ if $r > \frac{1}{c}$ and $\epsilon_2>0$ provided that $r^2-r-1>0$. Accordingly, we have         
 \begin{align*}
        C(r+1)|Du|^{\gamma-1}\frac{|Dm|}{m^{r+\bar{\alpha}+1}}&\leq
        \frac{(c r - 1)}{r^2}\frac{|Du|^\gamma}{m^{r+\bar{\alpha}}}+ \frac{C(r+1)}{2}\frac{|Dm|^2}{m^{r+2}}+  C(r+1)C_{r}^3\frac{1}{m^{r+\delta}}\\
        &=\frac{r^2-r-1}{r(r+1)}\left|D\frac{1}{m^{r/2}}\right|^2 + \frac{(4 c r - 1)}{4  r^2 }\frac{|Du|^\gamma}{m^{r+\bar{\alpha}}}+ C_{r}^3\frac{1}{m^{r+\delta}}
        \end{align*}
        for some suitable constant, $C_r^3$.      
From the previous estimates and after multiplying both sides of \eqref{fulleq1} by $2r^2$, we have,
\begin{align*}
&r^2\int_{\Tt^d}\frac{1}{m^{r+1}}dx+
2r\int_{\Tt^d}\left|D\frac{1}{m^{r/2}}\right|^2dx+2\int_{\Tt^d}\frac{|Du|^\gamma}{m^{r+\bar\alpha}}dx\\
&\hspace{2cm}\leq
2r^2C_r^3\int_{\Tt^d}\frac{1}{m^{r+\delta}}dx+2r^2C_r^2.
\end{align*}

\end{proof}
\begin{cor}\label{1mInLr}
        Suppose that Assumptions \ref{A2}-\ref{alphaexponent} hold. Then, there exists a constant, $C_r>0$,  such that, for any solution, $(u,m)$, of  Problem
\ref{P1},
 we have $\|\frac 1 m\|_{L^r(\Tt^d)}\leq C_r$.     
\end{cor}
\begin{proof}
        From Proposition \ref{firstlbPro}, we have
        \begin{align}\label{eq:newlbound}
            \int_{\Tt^d}\frac{1}{m^{r+1}}dx \leq C_r\int_{\Tt^d}\frac{1}{m^{r+\delta}}dx+C_r
            \end{align}
            where $\delta =\frac{2\bar{\alpha}}{(2-\gamma)}$.
Because Assumption \ref{alphaexponent} holds, $\delta<1$. Thus,  using Young's inequality, we have
            \[
            \frac{C_r}{m^{r+\delta}}\leq \frac{1}{2}\frac{1}{m^{r+1}} +C_r^3.
            \]
       
             Therefore, from \eqref{eq:newlbound}, we have,
            \[
            \left\|\frac{1}{m^{r+1}}\right\|_{L^{r+1}(\Tt^d)} \leq [2(C_{r}^2+C_r^3)]^{\frac{1}{r+1}}.
            \]

\end{proof}

\section{Improved regularity}
\label{impreg}

Now, we build upon the preceding results to get estimates for $u$\ in $W^{2,p}$. Once these estimates are made, we obtain bounds for $u$ and  $m$\ in any Sobolev space and uniform lower bounds for $m$.

\begin{lem}\label{lemmFP}
        Let $g:\Tt^d \to \Rr^d$ be a $C^\infty$ vector field.  Assume that, for  some $p_0>d$, there exists $C>0$ such that $\|g\|_{L^{p_0}(\Tt^d)}\leq C$.
Then,  for each $q>1$, 
there exists a constant, $C_q>0$, such that, for any non-negative solution, $w:\Tt^d\to \Rr,$  of the Fokker-Planck equation,
        \begin{equation}\label{eq:fokker-planck_w}
        w(x)-\Delta w(x) -\div(g(x)w(x))=1,
        \end{equation}
                we have $\|w\|_{L^q(\Tt^d)}\leq C_q$.

\end{lem}
\begin{proof}
        First, we multiply \eqref{eq:fokker-planck_w} by $w^{q-1}$. Integrating by parts, we get
        \[
        \int_{\Tt^d}w^q dx + (q-1)\int_{\Tt^d}w^{q-2}|Dw|^2dx+(q-1)\int_{\Tt^d}w^{q-1}g\cdot Dw dx = \int_{\Tt^d} w^{q-1}dx.
        \]
        Using Cauchy's and Young's inequalities, we obtain the estimate
        \begin{equation}\label{eq:est_FP_CauYou}
        \int_{\Tt^d}w^qdx+\int_{\Tt^d}|Dw^{\frac q 2}|^2dx\leq C_q\int_{\Tt^d} |g|^2w^q dx +C_q.
        \end{equation}
        Next, using Sobolev's inequality, we have        \begin{align*}
        \|w^q\|_{L^{\frac{2^*}{2}}(\Tt^d)}&\leq C_q\int_{\Tt^d} |g|^2w^qdx+C_q\\ 
        &\leq \tilde{C}_q\left( \||g|^2\|^r_{L^{p_1}(\Tt^d)}+\|w^q\|^{r'}_{L^{p_1'}}\right)+C_q
        \end{align*}
        for any $r>1$, where $p_1=\frac{p_0}{2}$ and the conjugate exponents $r'$ and $p_1'$ satisfy $\frac 1 r +\frac{1}{r'}=1$ and $\frac{1}{p_1}+\frac{1}{p_1'}=1$, respectively. Because $p_1>\frac d 2$, we have $p_1'<\frac{2^*}{2}$. Moreover, because $w\geq 0$, by integrating \eqref{eq:fokker-planck_w}, we have $\|w\|_{L^1(\Tt^d)}=1$.

Therefore, by interpolation,
        \[
        \|w\|_{L^{qp_1'}(\Tt^d)}\leq \|w\|^\theta_{L^{q\frac{2^*}{2}}(\Tt^d)}\|w\|^{1-\theta}_{L^1(\Tt^d)}=\|w\|^\theta_{L^{q\frac{2^*}{2}}(\Tt^d)}
        \]
        for some $0<\theta<1$. Hence, $\|w^q\|_{L^{p_1'}(\Tt^d)}\leq \|w^q\|^\theta_{L^{\frac{2^*}{2}}(\Tt^d)}$.
        Combining the previous bounds, we have the estimate
        \[
\|w^q\|_{L^{\frac{2^*}{2}}(\Tt^d)}\leq \tilde{C}_q\left( \||g|^2\|^r_{L^{p_1}(\Tt^d)}+\|w^q\|^{r'\theta}_{L^{\frac{2^*}{2}}(\Tt^d)}\right)+\tilde C_q.
        \]
        Finally, by choosing a large enough $r$ such that $r'\theta <1$, we obtain $\|w^q\|_{L^\frac{2^*}{2}(\Tt^d)}\leq C_q$. To end the proof, we observe that, from \eqref{eq:est_FP_CauYou}, it follows that, for any $q>1$, there exists $C_q>0$ such that $\|w\|_{L^q(\Tt^d)}\leq C_q$.
\end{proof}

\begin{pro}\label{DuLgamma}
  Suppose that Assumptions  \ref{A2}-\ref{alphaexponent} hold. Then, for any $1<p<\infty$, there exists a
 constant, $C>0$,  such that, for any solution, $(u,m)$, of  Problem
 \ref{P1}, we have 
 \begin{equation}
 \label{e1}
 \|D^2u\|_{L^p(\Tt^d)}\leq C \quad 
 \end{equation}
 and
 \[
 \|m\|_{L^p(\Tt^d)}\leq C.
 \]
 Furthermore,
 \begin{equation}
 \label{e2}
 \|Du\|_{L^{2p}(\Tt^d)}\leq C. 
 \end{equation}
\end{pro}
\begin{proof}
From the Gagliardo-Nirenberg interpolation inequality, we obtain
\[
\|Du\|_{L^{2p}(\Tt^d)}\leq C\|D^2u\|^{\frac{1}{2}}_{L^p(\Tt^d)}\|u\|_{L^\infty(\Tt^d)}^{\frac 1 2}
\]
for $1<p<\infty$. 
Therefore, from Corollary \ref{uLinfty}, we have the bound
\[
\|Du\|^\gamma_{L^{2p}(\Tt^d)}\leq C\|D^2u\|^{\frac{\gamma}{2}}_{L^p(\Tt^d)}.
\]
Next, we write the first equation in \eqref{eq1} as
\[
u-\Delta u=f,
\]
where 
\[
f=-m^\alpha H\left(x,\frac{Du}{m^\alpha}\right) -V.
\]
From Assumption \ref{As1.2} and standard elliptic regularity, we have
\[
\|D^2u\|_{L^p(\Tt^d)}\leq C\|m^\alpha\|_{L^{p}(\Tt^d)}+C\left\|\frac{|Du|^\gamma}{m^{\bar\alpha}}\right\|_{L^{p}(\Tt^d)}+\|V\|_{L^p(\Tt^d)}.
\]
 According to Assumption \ref{boundV}, $V$\ is bounded.  Thus, $V\in L^p(\Tt^d)$ for all $1<p<+\infty$.
Using Holder's inequality, we have
\begin{align*}
\left\|\frac{|Du|^\gamma}{m^{\bar\alpha}}\right\|_{L^{p}(\Tt^d)}&\leq \left\|\frac{1}{m^{\bar{\alpha}}}\right\|_{L^s(\Tt^d)}\||Du|^\gamma\|_{L^{\tilde s}(\Tt^d)}\\
&\leq \left\|\frac{1}{m}\right\|_{L^{\bar{\alpha}s}(\Tt^d)}^{\bar{\alpha}}
\|Du\|_{L^{\gamma \tilde s}(\Tt^d)}^\gamma,
\end{align*}
where $s>p$ such that $\gamma \tilde s<2p$ and $\frac{1}{p}  = \frac 1 s + \frac{1}{\tilde s}$. Thus, using Corollary \ref{1mInLr}, we see that there exists a constant, $C>0$, such that 
\[
\left\|\frac{|Du|^\gamma}{m^{\bar\alpha}}\right\|_{L^p(\Tt^d)}\leq C \|Du\|^\gamma_{L^{\gamma\tilde s} (\Tt^d)}.
\]
Therefore, taking into account that  $\gamma \tilde s<2p$ , we have
\begin{align*}
\|D^2u\|_{L^p(\Tt^d)}&\leq  C\|Du\|^\gamma_{L^{\gamma\tilde s}
(\Tt^d)}+C\|m^\alpha\|_{L^p(\Tt^d)}+C\\
&\leq C\|D^2u\|^{\frac{\gamma}{2}}_{L^p(\Tt^d)}+C\|m^\alpha\|_{L^p(\Tt^d)}+C.
\end{align*}
Since $\frac{\gamma}{2}<1$, we obtain
\begin{align*}
\|D^2u\|_{L^p(\Tt^d)}
&\leq C\|m^\alpha\|_{L^p(\Tt^d)}+C.
\end{align*}
The above arguments also imply that
\[
\|Du\|_{L^{2p}(\Tt^d)}\leq C\|D^2u\|^{\frac 1 2 }_{L^p(\Tt^d)}\leq C\|m^\alpha\|^{\frac 1 2}_{L^p(\Tt^d)}+C
\]
for all $p>1$. Thus, it remains for us to show that the right-hand side of the preceding estimate is bounded.

From Proposition \ref{L9}, we have $\|m^{(\alpha+1)\frac{2^*}{2}}\|_{L^1(\Tt^d)}\leq C$. Thus, there exists $p_0>1$ such that $\|m^\alpha\|_{L^{p_0}(\Tt^d)}\leq C$.
Consequently, $\|Du\|_{L^{2p_0}(\Tt^d)} \leq C$ for any $p_0\leq \frac{(\alpha+1) d}{\alpha (d-2)}$.
Using Assumption \ref{A6} and Holder's inequality, we have
\[
\|D_pH\|_{L^p(\Tt^d)}=\left\|\frac{|Du|^{\gamma-1}}{m^\alpha}\right\|_{L^p(\Tt^d)}\leq \||Du|^{\gamma-1}\|_{L^{p_1}(\Tt^d)}\left\|\frac{1}{m^{\bar{\alpha}}}\right\|_{\tilde p_1},
\]
for $\frac{1}{p}=\frac{1}{p_1}+\frac{1}{\tilde p_1}$. Thus, using Corollary \ref{1mInLr} in the preceding estimate, we see that there exists a constant, $C:=C_{p_1}>0$, such that
\[
\left\|\frac{|Du|^{\gamma-1}}{m^\alpha}\right\|_{L^p(\Tt^d)}\leq C \||Du|^{\gamma-1}\|_{L^{p_1}(\Tt^d)}=C \|Du\|^{\gamma-1}_{L^{p_1(\gamma-1)}(\Tt^d)}.
\]
Let  $p_1=\frac{2p_0}{\gamma-1}$. If $p_1>d$, we can select large enough  $\tilde p_1$
such that $p>d$. In this case, we have $\|D_pH\|_{L^p(\Tt^d)}\leq C$ for some $p>d$.  Finally, because of the estimate for $D_pH$, Lemma \ref{lemmFP} implies that $\|m\|_{L^q(\Tt^d)}\leq C_q$ for any $q>1$. Hence, \eqref{e1} and \eqref{e2} hold for any $p>1$. The  existence of an exponent $p_1>d$ is ensured by  the condition
\[
\frac{2 (\alpha+1) }{(\gamma-1) \alpha (d-2)}>1;
\]   
that is, if $d\leq 4$ or $\gamma\leq\frac{d}{d-2}$, the condition holds for any $\alpha>0$. Otherwise, for $d>4$ and $\gamma>\frac{d}{d-2}$, the constraint in $\alpha$ is $\alpha<\frac{2}{d (\gamma-2)-d}$. 

These conditions are implied by Assumption \ref{alphaexponent}.  
\end{proof}

\begin{cor}\label{CorDuLinfty}
        Suppose that Assumptions \ref{A2}-\ref{alphaexponent} hold. Then, there exists a
        constant, $C=C(\gamma,\alpha)>0$,  such that, for any solution, $(u,m)$, of  Problem
        \ref{P1} we have 
        \begin{align*}
        \|Du\|_{L^\infty(\Tt^d)}&\leq C.
        \end{align*}
\end{cor}

\begin{proof}
        From Lemma \ref{lemmFP} and Proposition \ref{DuLgamma}, we have $Du\in W^{1,p}(\Tt^d)$ for all $p>1$. Using Morrey's inequality for $p>d$, we have 
        \[
        \|Du\|_{C^{0,\beta}(\Tt^d)}\leq C \|Du\|_{W^{1,p}(\Tt^d)}
        \]
        for $\beta=1-\frac d p$. The preceding estimate implies the result.
\end{proof}

\begin{pro}\label{mInLinfty}
Suppose that Assumptions \ref{A2}-\ref{alphaexponent} hold. Then, there exists a
constant, $C>0$,  such that, for any solution, $(u,m)$, of  Problem \ref{P1}, we have $\|m\|_{L^\infty(\Tt^d)},\|Dm\|_{L^\infty(\Tt^d)}\leq C$.  Moreover, $\left\|\frac 1 m \right\|_{L^\infty(\Tt^d)}\leq C$. 

\end{pro}
\begin{proof}
        
From Corollary \ref{1mInLr}, Proposition \ref{DuLgamma} and Corollary \ref{CorDuLinfty}, there exist functions, $a:\Tt^d\to \Rr^d$ and $b:\Tt^d\to \Rr$, bounded in $L^p(\Tt^d)$ for every $p>1$, such that         
\begin{equation*}
\Delta m(x) =a(x)\cdot Dm(x)+b(x)m(x).\\ 
\end{equation*}
Let $w=\ln m$. Then,
\begin{equation*}
\Delta w(x) +|Dw|^2-a(x)\cdot Dm(x)=b(x).\\ 
\end{equation*}
Next, we use the adjoint method as in \cite[Proposition 6.9]{GPatVrt}, to obtain
 \[
 \|w\|_{L^\infty(\Tt^d)},\|Dw\|_{L^\infty(\Tt^d)}\leq C.
 \]
 Thus, because $\ln m$ is Lipschitz and $\int_{\Tt^d}m=1$, we have that both $m$ and $\frac 1 m$ are bounded. Finally, since $Dw=D\ln m=\frac{Dm}{m}$, we have that $Dm=mDw$ is in $L^\infty$.
\end{proof}

\begin{pro}\label{aprioriDuDm}
        Suppose that Assumptions \ref{A2}-\ref{alphaexponent} hold. Then, there exists a
        constant, $C:=C(k,l,p)>0$, $(l,k)\in\Nn\times \Nn$, $p>1,$  such that, for any solution, $(u,m)$, of  Problem \ref{P1}, we have 
        \[
        \|D^ku\|_{L^p(\Tt^d)},\|D^lm\|_{L^p(\Tt^d)}\leq C.
        \]
        
\end{pro}
\begin{proof}
        This proof is a simple bootstrapping argument. First, we use the regularity given by Corollary \ref{1mInLr}, Proposition \ref{DuLgamma}, Corollary \ref{CorDuLinfty}, and Proposition \ref{mInLinfty}. Finally, we use the elliptic regularity repeatedly on the equations for $u$, $m$ and their derivatives. 
\end{proof}

\section{Proof of Theorem \ref{mtheo}}\label{existence_cont_meth}
To prove Theorem \ref{mtheo} and establish the existence of a solution, we use the continuation method.
For that,  we  consider the following problem:
\begin{equation}\label{eqcont1}
\begin{cases}
&u_\lambda-\Delta u_\lambda +m_\lambda^{\alpha}H_\lambda\left(x,\frac{Du_\lambda}{m_\lambda^{\alpha}}\right) + V_\lambda(x,m) = 0\\ 
&m_\lambda -\Delta m_\lambda - \div {\left( D_pH_\lambda\left( x,\frac{Du_\lambda}{m_\lambda^{\alpha}}\right)m_\lambda\right)} =1,
\end{cases}
\end{equation}
where $H_\lambda(x,p)=\lambda H+(1-\lambda)(1+|p|^2)^{\frac{\gamma}{2}}$, $V_\lambda = \lambda V+(1-\lambda) \arctan(m)$ and $0\leq \lambda\leq 1$. 
For $\lambda=0,$ the problem has the following solution:
\begin{equation}
\label{u0m0}
u_0= -\left(1+\frac{\pi}{4}\right), \qquad m_0= 1.
\end{equation}
This solution is unique thanks to the monotonicity given by Assumptions \ref{Vincreasing} and \ref{A1}. For $\lambda =1$, \eqref{eqcont1} reduces to \eqref{eq1}.

For $k\in
\Nn,$ consider the Hilbert space $E^k:= H^k(\Tt^d)\times H^k(\Tt^d)$ with norm 
\[
\|w\|_{E^k}^2=\|v\|_{H^k(\Tt^d)}^2+\|f\|_{H^k(\Tt^d)}^2,
\]
where $w=(v,f)\in E^k$ and $E^0:=L^2(\Tt^d)\times L^2(\Tt^d)$.
By the Sobolev Embedding Theorem for $k>\frac d 2$, we have $H^k(\Tt^d)\in
C^{0,\xi}(\Tt^d)$ for some $\xi\in(0,1)$.

A classical solution to \eqref{eqcont1} is a pair $(u_\lambda,m_\lambda)\in\bigcap_{k\geq
0}E^k$. We fix  $k_0>\frac{d}{2}$ and define a map,
$F:[0,1]\times E^{k_0+2}\to E^{k_0}$, by

\[
F(\lambda,\nu):=
\left(
\begin{matrix}
u-\Delta u +m^{\alpha}H_\lambda \left( x, \frac{Du}{m^{\alpha}} \right) +
V_\lambda(x,m)\\ 
m -\Delta m - \div {\left( D_pH_\lambda \left( x,\frac{Du}{m^{\alpha}}\right)m\right)}
-1\\
\end{matrix}
\right), 
\]
where $\nu=(u,m)$. Accordingly, we write \eqref{eqcont1} as
\[
F(\lambda,\nu_\lambda)=0,
\]
where $\nu_\lambda=(u_\lambda,m_\lambda)$. Moreover, as remarked earlier,
$F(0,\nu_0)=0$ has only the trivial solution $(u_0, m_0)$.
Notice that for any $\zeta >0$, the map $F:E^{k_0+2}\cup \{m>\zeta \}\to
E^{k_0}$ is $C^{\infty}$. This holds because $H^k(\Tt^d)$ is an algebra,  for $k>\frac d 2$. Moreover, using a standard bootstrapping argument as in Proposition \ref{aprioriDuDm} and the bounds in the preceding sections, we see that whenever $(u_\lambda,m_\lambda)\in E^{k_0+2}   $ solves \eqref{eqcont1} with $m_\lambda>0$ then $(u_\lambda,m_\lambda)\in E^k$ for all $k$ and, hence, it is a classical solution.

Next,  we define the set 
\begin{equation}
\label{lset}
\Lambda=\{\lambda \in [0,1]\  |\  \eqref{eqcont1} \mbox{ has a solution  } (u,m)\in E^{k_0+2} \mbox{  with } m>0\}.
\end{equation}
To prove Theorem \ref{mtheo}, we show that $\Lambda$ is relatively open and closed. Consequently, $\Lambda = [0,1]$ and, in particular, \eqref{eq1} has a solution. First, in the next proposition, we show that $\Lambda$ is closed. 

\begin{pro}
\label{lclose}
Suppose that Assumptions \ref{A2}-\ref{alphaexponent} hold. Then, the set $\Lambda$ in \eqref{lset} is closed. 
\end{pro}
\begin{proof}
To prove this proposition, we show that for any sequence, $\lambda_n\in\Lambda,$ such that $\lambda\to \bar\lambda$ as $n\to \infty$, we have $\bar\lambda \in\Lambda$. We fix a sequence $\lambda_n$ and corresponding solutions, $(u_{\lambda_n},m_{\lambda_n}),$ to \eqref{eqcont1}. Because the bounds in Proposition \ref{aprioriDuDm} are independent of $n\in \Nn$, and by taking a subsequence, if necessary, we can assume that $(u_{\lambda_n},m_{\lambda_n})\in E^{k_0+2}$. Moreover, by using Propositions \ref{mInLinfty} and \ref{aprioriDuDm},  we have $m_{\lambda_n}^{-1}\to m^{-1}\in C(\Tt^d)$. Therefore, taking the limit in \eqref{eqcont1}, we see that $(u,m)$ solves \eqref{eqcont1} for $\lambda=\bar\lambda$. Thus, $\bar\lambda\in\Lambda$. 
\end{proof}

To show that $\Lambda$ is open, we use the implicit function theorem. For that, we recall that 
the partial derivative of $F$ in the second variable at $\nu_\lambda=(u_\lambda,m_\lambda)$,
\[
\mathcal L_\lambda =D_2F(\lambda,\nu_\lambda):E^k\to E^{k-2},
\]
is
\begin{equation*}
\mathcal{L}_\lambda (w) := 
\lim_{\epsilon\to 0}\frac{F(\lambda,u+\epsilon v,m+\epsilon f)-F(\lambda,u,m)}{\epsilon}=
\end{equation*}
\begin{equation*}
=\begin{bmatrix}
v - \Delta v +\alpha m^{\alpha-1} f\left(H_\lambda \left(x,\frac{Du}{m^\alpha}\right)-\frac{Du}{m^\alpha}\cdot
D_pH_\lambda \left(x,\frac{Du}{m^\alpha}\right)\right)+D_pH_\lambda \left(x,\frac{Du}{m^\alpha}\right)\cdot
Dv+D_m  V_\lambda f \\
\\
f-\Delta f-\div\left[D_pH_\lambda\left(x,\frac{Du}{m^\alpha}\right)f+m^{1-\alpha}D_{pp}^2H_\lambda\left(x,\frac{Du}{m^\alpha}\right)\cdot
Dv-\alpha fD_{pp}^2H_\lambda \left(x,\frac{Du}{m^\alpha}\right)\cdot \frac{Du}{m^\alpha}\right]\\
\end{bmatrix},
\end{equation*}
where $w=(v,f)\in E^k$. In principle, $\mathcal L_\lambda$ is only defined for large enough  $k$. However, by inspecting the coefficients, it is easy to see that  $\mathcal L_\lambda$  has a unique extension to $E^k$ for any $k>1$.

To show that $\Lambda$ is open, we show that $\mathcal L_\lambda$ is invertible and apply the implicit function theorem. To prove invertibility, we use an argument similar to the one in the proof of the Lax-Milgram theorem.  Let $E=E^1$. We set $Pw=(f,-v)$ for $w=(v,f)$. For $w_1,w_2\in E$, we define
\begin{equation*}
B_\lambda[w_1,w_2]=\int_{\Tt^d} \mathcal L_\lambda(w_1)\cdot Pw_2 dx.
\end{equation*}
For smooth $w_1$, $w_2$, integration by parts gives
\begin{align*}
&B_\lambda[w_1,w_2]=
\int_{\Tt^d}\Bigg[ \bigg[\alpha m^{\alpha-1} f_{1} \left(H_\lambda \left(x,\frac{Du}{m^\alpha}\right)-\frac{Du}{m^\alpha}\cdot
D_pH_\lambda \left(x,\frac{Du}{m^\alpha}\right)\right)\\& +D_pH_\lambda \left(x,\frac{Du}{m^\alpha}\right)\cdot Dv_{1} +D_m  V_\lambda f_1 \bigg] f_2\\&
-\bigg[D_pH_\lambda\left(x,\frac{Du}{m^\alpha}\right)f_1+m^{1-\alpha}D_{pp}^2H_\lambda\left(x,\frac{Du}{m^\alpha}\right)\cdot
Dv_1\\
&-\alpha f_1D_{pp}^2H_\lambda \left(x,\frac{Du}{m^\alpha}\right)\cdot \frac{Du}{m^\alpha}\bigg]\cdot Dv_2+Dv_1\cdot Df_2-Df_1\cdot Dv_2+v_1f_2-f_1v_2\Bigg]dx.
\end{align*}
From the results in Section \ref{impreg}, we have:
\begin{lem}
\label{BisBounded}
 Suppose that Assumptions \ref{A2}-\ref{alphaexponent} hold.  Let $\lambda\in \Lambda$. Then,       $B_\lambda$ is bounded; that is, there exists a constant, $C>0$, such that
        \[
        |B_\lambda [w_1,w_2]|\leq C\|w_1\|_E\|w_2\|_E
        \]
        for any $w_1,w_2\in E$.
\end{lem}
\begin{proof}
To prove this lemma, we use Holder's inequality on each summand together with  Corollary \ref{corboundDum} and Proposition \ref{aprioriDuDm}.      
\end{proof}

\begin{lemma}
\label{L63}
        Suppose that Assumptions \ref{A2}-\ref{alphaexponent} hold.  Let $\lambda\in
\Lambda$. There exists a bounded linear mapping, $A:E\to E$, such that $B_\lambda[w_1,w_2]=(Aw_1,w_2)_{E}$.
\end{lemma}
\begin{proof}
        This proof follows from Lemma \ref{BisBounded} and the Riesz Representation Theorem.
\end{proof}
\begin{lem}
\label{lem64}
Suppose that Assumptions \ref{A2}-\ref{A1} hold.  Let $\lambda\in
\Lambda$ and let  $A$ be as in the preceding Lemma. Then, there exists a constant, $C>0$, such that $\|Aw\|_E\geq C\|w\|_E$.  \end{lem}
\begin{proof}

We prove the Lemma by contradiction. If the claim were false, there would exist a sequence, $w_n\in E,$ such that $\|w_n\|_E=1$ and $Aw_n\to 0$. Accordingly, $B[w_n,w_n]\to 0$. Let $w_n=(v_n,f_n)$. Then, for  $Q=\frac{Du}{m^\alpha}$, we have    
\begin{align*}
\begin{split}
&B_\lambda[w_n,w_n]=\\&=\int_{\Tt^d}
\bigg[
-\alpha m^{\alpha-1} f_n^2
\left( D_pH_\lambda
\cdot Q-H_\lambda\right)-m^{1-\alpha}(Dv_n)^T\cdot D_{pp}^2H_\lambda \cdot Dv_n\\&
-\alpha f_n(Dv_n)^T\cdot D_{pp}^2H_\lambda \cdot Q +D_m  V_\lambda f_n^2\bigg]dx\\
&=\int_{\Tt^d}
\bigg[
-\alpha m^{\alpha-1} f_n^2
\left( D_pH_\lambda
\cdot Q-H_\lambda-\frac\alpha 4Q^T\cdot D_{pp}^2H_\lambda\cdot Q\right)\\&
-m^{\alpha-1}(m^{1-\alpha}Dv_n-\frac \alpha 2 f_n Q)^T D_{pp}^2 H_\lambda
\cdot (m^{1-\alpha}Dv_n-\frac \alpha 2 f_n Q)+D_m  V_\lambda f_n^2\bigg]dx
 \to 0.
\end{split}
\end{align*}
From the previous limit and Assumptions \ref{Vincreasing} and \ref{A1}, we conclude that $f_n\to 0$ and  $Dv_n\to 0$ in $L^2(\Tt^d)$. Next, we compute
\begin{align*}
(Aw_n,(f_n,0))=&B[w_n,(f_n,0)]\\=
&\int_{\Tt^d}\Bigg[
-\bigg[D_pH_\lambda\left(x,\frac{Du}{m^\alpha}\right)f_n+m^{1-\alpha}D_{pp}^2H_\lambda\left(x,\frac{Du}{m^\alpha}\right)\cdot
Dv_n\\
&-\alpha f_1D_{pp}^2H_\lambda \left(x,\frac{Du}{m^\alpha}\right)\cdot \frac{Du}{m^\alpha}\bigg]\cdot
Df_n-|Df_n|^2-f_n^2\Bigg]dx\to 0.
\end{align*}
Because $f_n, Dv_n\to 0$ in $L^2(\Tt^d)$, we have $Df_n\to 0$\ in  $L^2(\Tt^d)$. Finally, by computing  
\begin{align*}
(Aw_n,(0,1))=&B[w_n,(0,1)]\\=
&
\int_{\Tt^d}\Bigg[ \bigg[\alpha m^{\alpha-1} f_{n} \left(H_\lambda \left(x,\frac{Du}{m^\alpha}\right)-\frac{Du}{m^\alpha}\cdot
D_pH_\lambda \left(x,\frac{Du}{m^\alpha}\right)\right)\\& +D_pH_\lambda \left(x,\frac{Du}{m^\alpha}\right)\cdot
Dv_{n} +D_m  V_\lambda f_n \bigg]+v_n\Bigg]dx\to 0, 
\end{align*}
we conclude that $\int_{\Tt^d}v_n\to 0$. Therefore, $w_n\to 0$\ in $E$, which contradicts $\|w_n\|_E=1$.  
\end{proof}
\begin{lem}
\label{L65}
        Suppose that Assumptions \ref{A2}-\ref{A1} hold. Let $\lambda\in
\Lambda$. The range $R(A)$ is closed and $R(A)=E$.
\end{lem}
\begin{proof}
        Due to the bound in Lemma \ref{lem64}, we use, for instance,  the argument in \cite[Lemma 3.4]{GMit}  to conclude this Lemma.
\end{proof}
\begin{lem}
\label{L66}
        Suppose that Assumptions \ref{A2}-\ref{A1} hold.  Let $\lambda\in
\Lambda$. For any $w_0\in E^0$ there exists a unique $w\in E$ such that $B_\lambda[w,\tilde w]=(w_0,\tilde w)_{E^0}$ for all $\tilde w\in E$. This implies that $w$ is the unique solution to $\mathcal L_\lambda(w)=w_0$. Then, elliptic regularity implies that $w\in E^2$ and $\mathcal L_\lambda (w)=w_0$ in $E^2$.
\end{lem}
\begin{proof}
        Consider the functional $\tilde w\mapsto (w_0,\tilde w)_{E^0}$ on $E$. By the Riesz Representation Theorem, there exists $\omega\in E$ such that $(w_0,\tilde w)_{E^0}=(\omega,\tilde w)_{E}$. Let $w=A^{-1}\omega$. Then,
        \[
        B_\lambda[w,\tilde w]=(Aw,\tilde w)_E=(\omega,\tilde w)_E=(w_0,\tilde w)_{E^0}.       
        \]
Therefore, $v$ is a weak solution to
\[
v - \Delta v - \alpha m^{\alpha-1} f(H_\lambda (x,\frac{Du}{m^\alpha})-\frac{Du}{m^\alpha}\cdot
D_pH_\lambda (x,\frac{Du}{m^\alpha}))-D_pH_\lambda (x,\frac{Du}{m^\alpha})\cdot
Dv+D_m  V_\lambda f=v_0
\]
and $f$ is a weak solution to
\[
f-\Delta f-\div[D_pH_\lambda(x,\frac{Du}{m^\alpha})f+m^{1-\alpha}D_{pp}^2H_\lambda(x,\frac{Du}{m^\alpha})\cdot
Dv-\alpha fD_{pp}^2H_\lambda (x,\frac{Du}{m^\alpha})\cdot \frac{Du}{m^\alpha}]=f_0
\]
From standard elliptic regularity theory and bootstrap arguments, we conclude that $w=(v,f)\in E^2$ and, thus, that $\mathcal L_\lambda(w)=w_0$.

Consequently, $\mathcal L_\lambda$ is a bijective operator from $E^2$ to $E^0$. Hence, $\mathcal L_\lambda: E^k\to E^{k-2}$ is an injective operator for any $k\geq 2$. We claim that it is also surjective. To see this, take any $w_0\in E^{k-2}$. Then, there exists $w\in E^2$ such that $\mathcal L_\lambda(w)=w_0$. From regularity theory for elliptic equations and bootstrap arguments, we conclude that  $w\in E^k$. Therefore, the claim holds and, hence, $\mathcal L_\lambda: E^k\to E^{k-2}$ is bijective.
\end{proof}
\begin{lem}
\label{L67}
  Suppose that Assumptions \ref{A2}-\ref{A1} hold.    Let $\lambda\in
\Lambda$. Then,
        $\mathcal L_\lambda: E^k\to E^{k-2}$ is an isomorphism for any $k\geq 2$.
\end{lem}
\begin{proof}
        Since $\mathcal L_\lambda: E^k\to E^{k-2}$ is bijective, we just need to prove that it is a bounded linear mapping. The boundedness follows directly from the bounds on $u_\lambda$, $m_\lambda$ and the smoothness of $V$.
\end{proof}
\begin{pro}
\label{lopen}
          Suppose that Assumptions \ref{A2}-\ref{A1} hold.  Let $\lambda\in
\Lambda$. Then, the set $\Lambda$ is open.
\end{pro}
\begin{proof}
        Let $k>\frac{d}{2}+1$ so that $H^{k-1}(\Tt^d)$ is an algebra. From Lemma \ref{L67}, for each $\bar\lambda\in\Lambda$, the partial derivative, $\mathcal L_{\bar\lambda}=D_2F(\bar\lambda,\nu_{\bar{\lambda}}):E^k\to E^{k-2}$, is an isometry. Therefore, by the Implicit Function Theorem for Banach spaces, there exists a unique solution $\nu_\lambda\in E^k$ to $F(\lambda,\nu_\lambda)=0$, in the neighborhood, $U$, of $\bar{\lambda}$. Since $H^{k-1}(\Tt^d)$ is an algebra, by using a bootstrapping argument, we get that  $u_\lambda$ and $m_\lambda $ are smooth. Therefore, $\nu_\lambda$ is a classical solution to \eqref{eqcont1}. Hence, $U\in \Lambda$ and we conclude that $\Lambda$ is open.
\end{proof}

The preceding results establish Theorem \ref{mtheo} as follows. 
\begin{proof}[Proof of Theorem \ref{mtheo}.]
Due to Propositions \ref{lclose} and \ref{lopen}, the set  $\Lambda$ is both open and closed. Because \eqref{u0m0} is a solution of \eqref{eqcont1} for $\lambda=0$, $\Lambda$ is non-empty. Consequently, $\Lambda=[0,1]$. 
\end{proof}

\def\polhk#1{\setbox0=\hbox{#1}{\ooalign{\hidewidth
			\lower1.5ex\hbox{`}\hidewidth\crcr\unhbox0}}} \def\cprime{$'$}

\end{document}